\documentclass[10pt]{article}

\usepackage{bm}
\usepackage[top=0.9in,bottom=1.1in,left=1.05in,right=1.05in]{geometry}

\usepackage{textgreek}
\usepackage{amsmath}
\usepackage{graphicx,subfig}
\usepackage[colorinlistoftodos]{todonotes}
\usepackage[colorlinks=true, allcolors=blue]{hyperref}
\usepackage{amsmath}
\usepackage{upgreek}
\usepackage{amssymb}
\usepackage{amsfonts}
\usepackage{amsthm}
\usepackage{mathrsfs}
\usepackage{fontenc}
\usepackage{empheq}
\usepackage{enumerate}
\usepackage{comment}
\usepackage{appendix}
\usepackage{bbm}

\usepackage{booktabs}
\usepackage{mathtools}
\mathtoolsset{showonlyrefs}

\newtheorem{remark}{Remark}

\newtheorem{assumption}{Assumption}

\newcommand{\hP}{\hat\dbP}
\newcommand{\ud}{\,\mathrm{d}}
\newcommand{\Var}{\mathrm{Var}}
\newcommand{\half}{\frac{1}{2}}
\newcommand{\transpose}{^{\operatorname{T}}}

\newcommand{\ba}{\begin{array}}
\newcommand{\ea}{\end{array}}
\newcommand{\be}{\begin{equation}}
\newcommand{\ee}{\end{equation}}
\newcommand{\bee}{\begin{equation*}}
\newcommand{\eee}{\end{equation*}}
\newcommand{\bea}{\begin{eqnarray}}
\newcommand{\eea}{\end{eqnarray}}
\newcommand{\beaa}{\begin{eqnarray*}}
\newcommand{\eeaa}{\end{eqnarray*}}

\def\dbE{\mathbb{E}}

\def\cF{{\cal F}}

\def\cN{{\cal N}}

\def\cT{{\cal T}}

\def\hE{\mathbb{E}}

\def\hP{\mathbb{P}}

\def\hR{\mathbb{R}}

\newcommand{\basa}{\begin{assumption}}
\newcommand{\easa}{\end{assumption}}

\newcommand{\bas}{\begin{assum}}
\newcommand{\eas}{\end{assum}}

\def\1{{\bf 1}}

\def\:{\!:\!}

at 9pt

\begin{document}
\newcommand{\rh}[1]{\textcolor{blue}{\textsf{[RH: #1]}}}
\newcommand{\zz}[1]{\textcolor{red}{\textsf{[ZZ: #1]}}}
\newtheorem{thm}{Theorem}[section]
\newtheorem{lem}[thm]{Lemma}
\newtheorem{cor}[thm]{Corollary}
\newtheorem{prop}[thm]{Proposition}
\newtheorem{rem}[thm]{Remark}
\newtheorem{eg}[thm]{Example}
\newtheorem{defn}[thm]{Definition}
\newtheorem{assum}[thm]{Assumption}

\renewcommand {\theequation}{\arabic{section}.\arabic{equation}}
\def\thesection{\arabic{section}}

\title{Convergence of the Backward Deep BSDE Method with Applications to Optimal Stopping Problems}

\author{Chengfan Gao\thanks{School of Economics, Fudan University, Shanghai, China, 200433, {\em 19300680118@fudan.edu.cn}. }\and Siping Gao \thanks{School of Economics, Fudan Univeristy, Shanghai, China, 200433, {\em 19307110436@fudan.edu.cn}.} \and Ruimeng Hu\thanks{Department of Mathematics, and Department of Statistics and Applied Probability, University of California, Santa Barbara, CA, 93106-3080, {\em rhu@ucsb.edu}.} \and  Zimu Zhu \thanks{Corresponding author. Departments of Statistics and Applied Probability, University of California, Santa Barbara, CA, 93106-3110, \em{zimuzhu@ucsb.edu}.}}
\date{\today}
%\date{\today}
\maketitle

\begin{abstract}
%The recently proposed machine learning algorithm, deep BSDE method [Han, Jentzen and E, PNAS, 115(34):8505-8510, 2018] has shown great performance in solving high-dimensional forward-backward stochastic differential equations (FBSDEs), and has inspired many followup works. In their design, the backward stochastic differential equation (BSDE) will be simulated in a forward manner, which is unfortunately not compatible with optimal stopping problems that usually require running the BSDE backwardly. Optimal stopping problems have broad applications in financial markets, such as pricing American and Bermudan options. 

%The idea is try to guess the initial value of the backward process $Y_t$ and the adjoint process $Z_t$ so as to meet the terminal condition of $Y_T$. 

%The seminal paper Han, Jentzen and E \cite{han2017deep,han2018solving} propose a new machine learning method to numerically compute high dimensional parabolic PDE. However, their numerical scheme is not adapted to optimal stopping problem, which is quite common in financial markets. 

The optimal stopping problem is one of the core problems in financial markets, with broad applications such as pricing American and Bermudan options. The deep BSDE method [Han, Jentzen and E, PNAS, 115(34):8505-8510, 2018] has shown great power in solving high-dimensional forward-backward stochastic differential equations (FBSDEs), and inspired many applications. However, the method solves backward stochastic differential equations (BSDEs) in a forward manner, which can not be used for optimal stopping problems that in general require running BSDE backwardly. To overcome this difficulty, a recent paper [Wang, Chen, Sudjianto, Liu and Shen, arXiv:1807.06622, 2018] proposed the backward deep BSDE method to solve the optimal stopping problem. In this paper, we provide the rigorous theory for the backward deep BSDE method. Specifically, 1. We derive the {\it a posteriori} error estimation, i.e., the error of the numerical solution can be bounded by the training loss function; and; 2. We give an upper bound of the loss function, which can be sufficiently small subject to universal approximations. We give two numerical examples, which present consistent performance with the proved theory. 

%are presented in high-dimensions, which show

%To fill the gap, a recent paper [Wang, Chen, Sudjianto, Liu and Shen, arXiv:1807.06622, 2018] proposed a new numerical scheme (which we call backward deep BSDE method) which is capable of computing optimal stopping problems. In this paper, we aim to provide a theoretical proof for the backward deep BSDE method: (1) A posteriori error estimation of the solution is derived, that is, the error of the numerical solution could be bounded by the training loss function; and (2) An upper bound of the loss function is provided, which could be sufficiently small subject to universal approximations. Several numerical examples are presented in high-dimensions for further illustrating the algorithm's efficiency and accuracy. 
\end{abstract}

%\textbf{Key words}: backward deep BSDE Method, decoupled FBSDE , machine learning, optimal stopping problem, Bermudan options.

\\


\section{Introduction}\label{sec:intro}

In financial markets, one of the core problems is to price derivatives such as Bermudan options or American options. When the number of underlying assets is large, it leads to  a high-dimensional optimal stopping problem. Deep learning methods have unlocked the possibility of solving high-dimensional problems with complex structures, which is not feasible with traditional numerical methods. Examples include recent breakthrough machine learning methods for high-dimensional nonlinear partial differential equations (PDEs) and backward stochastic differential equations (BSDEs) \cite{han2017deep,han2018solving}, and machine learning methods on stochastic control and differential games \cite{carmona2019convergence,hu2021deep}; see a review paper \cite{hu2022recent} and the references therein. 

However, the aforementioned deep learning algorithms (e.g., \cite{han2017deep,han2018solving}) for BSDEs can not be used to solve optimal stopping problems. Due to the nature of their algorithm design,  the BSDEs are simulated in a forward manner, while optimal stopping problems usually require running BSDEs backwardly. Therefore, new machine-learning methods are needed to address this issue. For example, in \cite{wang2018deep}, the authors proposed the backward deep BSDE method for solving the LIBOR market model with application to Bermudan options. Moreover, there have been other papers that utilize deep learning (DL) to compute optimal stopping problems and/or Bermudan options, e.g., \cite{bayraktar2022deep,becker2019deep,becker2020pricing,herrera2021optimal,hu2020deep,hure2020deep,kohler2010pricing,lapeyre2021neural,reppen2022deep,reppen2022neural}. Before the introduction of DL, conventional numerical methods for BSDE or forward-backward stochastic differential equations (FBSDEs) include  \cite{bally2005quantization,bender2007forward,bender2008time,bouchard2004discrete,crisan2012solving,gobet2005regression,guo2015monotone,zhang2004numerical,zhao2010stable}, among others.

Despite the success of the developed machine learning methods, only limited work has been found on the theoretical support for the algorithms. To name a few, Han and Long \cite{han2020convergence} proved the convergence of the deep BSDE algorithm proposed in \cite{han2017deep,han2018solving}. Hur\'{e}, Pham and Warin \cite{hure2020deep} provided the convergence analysis for the deep backward dynamic programming method (DBDP), another DL algorithm for solving high-dimensional nonlinear PDEs and BSDEs. Han, Hu and Long \cite{han2020convergence1} proved the convergence of deep fictitious play (DFP) for the computing Nash equilibrium of the N-player asymmetric stochastic differential games.

%It is worth mentioning that a very recent paper by Andersson, Andersson and Oosterlee \cite{andersson2022convergence} uses a similar loss function and provide convergence analysis when the driver of the backward process does not depend on $Y$.

In this paper, we study the convergence of the algorithm proposed by \cite{wang2018deep}. The main difficulty of the proof is that the backward numerical process $Y^{\pi}$ is not be adapted to the filtration, {which is the natural filtration of the Brownian motion; see Section \ref{sec:method} for details}. Therefore, some existing estimates of BSDE (e.g., in \cite{bender2008time,han2020convergence}) can not be directly applied. Nevertheless, we are able to provide {\it a posteriori} error estimation of the solution and show that the error of the numerical solution can be bounded by the training loss. Then, we show that the training loss can be sufficiently small subject to sufficient smoothness of the corresponding PDE and universal approximations. 

The rest of this paper is organized as follows. In Section~\ref{sec:method}, we review the backward deep BSDE method developed by Wang \emph{et al.} \cite{wang2018deep}. Section~\ref{sec:analysis} introduces basic assumptions of the BSDEs and provides the convergence result of the deep BSDE method. We then make the connections between the algorithm and optimal stopping problems through the example of Bermudan option pricing in the financial market in Section~\ref{sec:derivative}. Numerical examples will be given in Section~\ref{sec:numerics} and we conclude in Section~\ref{sec:conclusing}.

\medskip
\noindent {{\bf Related work}: Several machine learning methods have been proposed for solving BSDEs and optimal stopping problems. For example, \cite{becker2019deep,becker2020pricing} solve optimal stopping problems by parameterizing the stopping rules, while \cite{hu2020deep} parameterizes the stop/continuation regions. Others, such as \cite{reppen2022deep,reppen2022neural}, parameterize the optimal stopping boundary, \cite{kohler2010pricing,lapeyre2021neural} approximate the continuation value using neural networks, and \cite{hure2020deep} solves the corresponding reflected BSDE. Finally, \cite{bayraktar2022deep} focuses on path-dependent American options and solves the corresponding path-dependent FBSDE.

Recently, \cite{herrera2021optimal} proposed a novel approach to approximating the continuation value using randomized neural networks with convergence guarantees, where hidden layers are generated randomly and only the last layer is trained. In Section~\ref{sec:numerics}, we use their numerical results as benchmarks. Another related work with convergence analysis is \cite{andersson2022convergence}, which uses a loss function similar to \eqref{def:BdeepBSDE-goal} but replaces the simulation scheme \eqref{def:BdeepBSDE-Yt} with a more costly scheme involving backward and forward simulations. There are several differences between our work and theirs: (1) the FBSDE analyzed in \cite{andersson2022convergence} has a driver that does not depend on $Y$, while our FBSDE does, and this feature is critical for optimal stopping problems; (2) their proof relies on assumptions on the Markov maps and regularity of the optimal map, while we use standard Lipschitz and polynomial growth conditions on the coefficients and regularity of the corresponding PDE solution; (3) their proof relies on strong convergence results and bistability for stochastic multistep methods, while we use error estimates from \cite{zhang2017backward}.

%Solving the BSDE numerically has been studied extensively in the literature, for instance, see \cite{zhang2004numerical}. When $d$ is large, several efficient deep learning algorithms have been proposed recently \cite{han2018solving,wang2018deep,liang2021deep}, to beat the curse of dimensionality. Among them, \cite{han2018solving} propose the simulate both $X$ and $Y$ in a forward manner by parameterizing $Z$ using neural networks and then minimize the difference between $g(X_T)$ and the simulated terminal states of $Y$. \rh{please add more ref.} 

\section{The Backward Deep BSDE Method}\label{sec:method}
Let $\left(\Omega, \cF, \{\cF_t\}_{0\leq t\leq T}, \hP\right)$ be a filtered probability space supporting a $d$-dimensional Brownian motion $W$ and $\cF=\left\{\cF_t\right\}_{0\leq t\leq T}$ is the natural filtration of $W$, augmented by all $\hP$-null sets. Consider the following generic BSDE on $\left[0,T\right]$:
\begin{equation}\label{def:BSDE}
\begin{dcases}
X_t=x+\int_0^t b(s,X_s)\ud s+\int_0^t \sigma(s,X_s)\ud W_s,\\
Y_t=g(X_T)+\int_t^T f(s,X_s,Y_s,Z_s)\ud s-\int_t^T Z_s\ud W_s,
\end{dcases}
\end{equation}
where $W_t\in \hR^d $, $X_t \in \hR^{d_1}$, $Y_t \in \hR^{d_2}$ and $Z_t \in \hR^{d_2\times d}$. For notation simplicity, we will carry out the discussions and proofs only for $d_2=1$, and the extension to multidimensional $Y$ is straightforward. More conditions on $b, \sigma, g, f$ will be specified in Section~\ref{sec:analysis}. We say the process $(X_t, Y_t, Z_t)$ is a solution to \eqref{def:BSDE} if it is $\cF_t$-adapted and satisfies proper integrability conditions (cf. \cite{zhang2017backward}).

Solving the BSDE numerically is not a easy problem especially when $d_1, d_2, d$ are large. In \cite{wang2018deep}, they propose the following backward deep BSDE scheme. Let $\pi$ be a partition of size $n$ on $[0, T]$: $0 = t_0 < t_1 < \ldots < t_n = T$, $t_i = ih$, $h = T/n$, the backward deep BSDE algorithm solve the minimization problem:
\begin{align}\label{def:BdeepBSDE-goal}
    & \inf_{\{\phi_i \in \cN_i\}_{i=0}^{n-1}} \Var\left[Y_0^\pi\right] \\
    & s.t. \quad X_{t_{i+1}}^{\pi}=X_{t_i}^{\pi}+b(t_i,X_{t_i}^{\pi})h+\sigma(t_i,X_{t_i}^{\pi})\Delta W_i, \quad X_0^{\pi}=x, \nonumber\\
    & \hspace{28pt} Y_{t_i}^{\pi} = Y^{\pi}_{t_{i+1}} + f(t_i,X_{t_i}^{\pi},Y_{t_{i+1}}^{\pi},Z_{t_i}^{\pi})h - Z_{t_i}^{\pi} \Delta W_i, \quad Y_T^{\pi}=g(X_T^{\pi}), \;
Z_{t_i}^{\pi}=\phi_i (X_{t_i}^\pi). \label{def:BdeepBSDE-Yt}
\end{align}
Here $X^\pi$ is firstly simulated forward given the initial value $x$ from 0 to $T$, then $Y^\pi$ is simulated backward from time $T$ to 0 given the terminal condition $g$ and the candidate of the parameterized adjoint process $Z_{t_i}^\pi = \phi_i(X_{t_i}^\pi)$. For brevity, we denote $\Delta W_i = W_{t_{i+1}} - W_{t_i}$ the Brownian increment, $\cN_i$  the hypothesis space related to deep neural networks, and use $0$ and $T$ as subscripts when time indices are $t_0$ or $t_n$.

In practice, the variance in \eqref{def:BdeepBSDE-goal} is approximated by Monte Carlo simulations, and the parameters in $\phi_i$ are optimized by stochastic gradient descent type algorithms. Intuitively, for deterministic initial value $x$, by the adaptivity to $\cF_t$,  $Y_0$ should be deterministic and thus has zero variance. The smaller the variance, the closer $Y_0^\pi$ to deterministic. This is indeed the rationality of choosing \eqref{def:BdeepBSDE-goal} as the loss function in training. This choice will be further justified in Section~\ref{sec:analysis}, where we show that, under proper assumptions, the error between $(Y^\pi, Z^\pi)$ in \eqref{def:BdeepBSDE-Yt} and the truth $(Y, Z)$ in \eqref{def:BSDE} can be bounded by the time discrete $h$ and $\Var[Y_0^\pi]$, and $\Var[Y_0^\pi]$ can be sufficiently small. 

\begin{rem}\label{rem:advantage}
The deep BSDE algorithm introduced in \cite{han2017deep,han2018solving}, and theoretically justified in \cite{han2020convergence}, is another very powerful deep learning algorithm in solving high-dimensional BSDEs or coupled FBSDEs. Compared to \eqref{def:BdeepBSDE-Yt}, the deep BSDE method simulates $Y^\pi$ in a forward manner after parameterizing $Y_0^\pi$ and $Z_{t_i}^\pi$ using neural networks, and minimizes $E|g(X_T^\pi) - Y_T^\pi|^2$ to obtain the optimal network parameters.

The advantage of \eqref{def:BdeepBSDE-goal}--\eqref{def:BdeepBSDE-Yt} is the capability of dealing with optimal stopping problems, due to its backward nature in simulating $Y$. This will be further explained in Section~\ref{sec:derivative}.
\end{rem}

\section{Convergence Analysis}\label{sec:analysis}
This section dedicates to the convergence analysis for the backward deep BSDE method reviewed in \eqref{def:BdeepBSDE-goal}--\eqref{def:BdeepBSDE-Yt}. More specifically, we shall first show that $\Var[Y_0^\pi]$ is a valid choice for loss functions by bounding the numerical errors using $\Var[Y_0^\pi]$. Thus, the smaller the variance is, the more accurate the numerical solution is. Secondly, we shall provide an upper bound for the minimized loss function, which can be sufficiently small subject to the universal approximation theorem.

We first introduce the standing assumptions and review some useful results in \cite{zhang2004numerical}.

\begin{assum}\label{assump:l2regularity}
Let $b, \sigma, f, g$ be deterministic functions:
$$
    \left(b, \sigma\right): \left[0,T\right] \times \hR^{d_1} \to \hR^{d_1} \times \hR^{d_1\times d}, \quad f: \left[0,T\right] \times \hR^{d_1} \times \hR^{d_2} \times \hR^{d_2\times d} \to \hR^{d_2}, \quad g: \hR^{d_1} \to \hR^{d_2}, \nonumber
$$
such that:
\begin{enumerate}

\item  $b\left(\cdot,0\right), \sigma \left(\cdot,0\right), f\left(\cdot,0,0,0\right)$ and $g\left(0\right)$ are bounded;

\item $b,\sigma,f,g$ are uniformly Lipchitz continuous in $\left(x,y,z\right)$ and $\half$-H\"{o}lder continuous in $t$,  with Lipchitz constant $K$.

%\item $b,\sigma,f$ are uniformly $H\ddot{o}lder-{1\over 2}$ continous in $t$ with Lipchitz constant $k$.
\item $d_1 = d$ and $\sigma \sigma\transpose \geq \delta I_{d_1}$ for some constant $\delta >0$.
\end{enumerate}

\end{assum}

\begin{assum}\label{assump:posteriorbdd}
Assume further that $f$ satisfies the following Lipchitz condition with respect to $y$ and $z$:
$$
    \left|f(t,x,y_1,z_1)-f(t,x,y_2,z_2)\right|^2\leq K_y\left|y_1-y_2\right|^2+K_z\left|z_1-z_2\right|^2 \nonumber
$$
where $K_y$ is sufficiently small $($at least $6T^2 K_y<1$$)$ and $K_z<{\frac{1}{4T}}$.
\end{assum}

\begin{assum}
\label{assump:FKformula}
Assume the following system of PDE:
$$
    \partial_t u^{i}+ b \cdot \nabla_x u^i + \half \text{Tr}\Big( \text{Hess}_xu^{i}(\sigma\sigma\transpose)\Big)+f^i\left(t,x,u, \partial_x u \sigma\right)=0, \; i=1,\cdots,d_2; \; u\left(T,x\right)=g\left(x\right). \nonumber
$$
has a classical solution $u  = \left[u^1, u^2, \ldots, u^{d_2}\right] \in C^{1,2}\big(\left[0,T\right]\times \hR^{d_1},\hR^{d_2}\big)$. 
\end{assum} 

\begin{assum}
\label{assump:minimizeVar}
Assume further that $f$ satisfies the following Lipchitz condition with respect to $y$:
$$
\left|f\left(t,x,y_1,z\right)-f\left(t,x,y_2,z\right)\right|^2\leq K_y\left|y_1-y_2\right|^2, \nonumber
$$
where $6T^2 K_y<1$.
\end{assum}

\begin{remark}
\begin{enumerate}
    \item  Under Assumption \ref{assump:l2regularity} and Assumption \ref{assump:FKformula}, the nonlinear Feynman-Kac formula holds and implies that $Y_t=u(t,X_t)$ and $Z_t=\nabla_x u\sigma(t,X_t)$, (See Theorem 5.1.4 in \cite{zhang2017backward}). Such representation of $Z$ will be used in Theorem \ref{easy case for upper bound}.
    
    \item Assumption \ref{assump:posteriorbdd} and Assumption \ref{assump:minimizeVar} are very similar, {and will be used correspondingly}, in Theorem \ref{thm:posteriorbdd} and Theorem \ref{easy case for upper bound}.
\end{enumerate}

\end{remark}

{Recall the partition $\pi$ on $\left[0,T\right]$ with  $0 = t_0 < t_1 < \ldots < t_n = T$,  $t_i = ih = iT/n$. For a given partition, let us define 
$$\pi(t) = t_i, \quad \bar \pi(t) = t_{i+1}, \quad \forall t \in [t_i, t_{i+1}).$$
We now recall the following $L^2$ regularity result. In the sequel, we will use $C$ as a generic positive constant, which may vary from line to line.}

\begin{prop}[{\cite[Theorem 3.1, Lemma 3.2]{zhang2004numerical}}]\label{prop:pathregularity}
Under Assumption~\ref{assump:l2regularity}, let $\left(X_t, Y_t, Z_t\right)$ be the solution to the (decoupled) FBSDE \eqref{def:BSDE}. The following estimate holds:
\begin{equation}\label{eq:pathregularity}
%\max_{0\leq i\leq n-1} \sup_{t_{i} <t \leq t_{i+1}} \dbE[|X_t-X_{t_{i}}|^2 + |Y_t-Y_{t_{i}}|^2\]+\sum_{i=0}^{n-1} \dbE\int_{t_{i}}^{t_{i+1}}|Z_t-Z_{t_i}|^2\ud t \leq Ch
\sup_{ t \in [0,T]} \dbE\left[\left|X_t-X_{\pi(t)}\right|^2 + \left|Y_t-Y_{\pi(t)}\right|^2\right]
+ \dbE\int_{0}^{T}\left|Z_t-Z_{\pi(t)}\right|^2\ud t \leq Ch,
\end{equation}
where $C$ may depend on $T, K$ but free of the partition parameter $h$.
\end{prop}

{Our} first theorem aims to justify the validity of using $\Var\left(Y_0^\pi\right)$ as training loss. To this end, we consider $\left\{X_{t_i}^\pi, Y_{t_i}^\pi\right\}_{i=0}^n$ satisfying
\begin{equation}\label{eq:BSDEdiscrete}
    \begin{dcases}
 X_{t_{i+1}}^{\pi}=X_{t_i}^{\pi}+b(t_i,X_{t_i}^{\pi})h+\sigma(t_i,X_{t_i}^{\pi})\Delta W_i, \quad X_0^{\pi}=x, \\
  Y_{t_i}^{\pi} = Y^{\pi}_{t_{i+1}} + f(t_i,X_{t_i}^{\pi},Y_{t_{i+1}}^{\pi},Z_{t_i}^{\pi})h - Z_{t_i}^{\pi} \Delta W_i, \quad Y_T^{\pi}=g(X_T^{\pi}), \;
    \end{dcases}
\end{equation}
where $Z_{t_i}^\pi$ is arbitrary but $\cF_{t_i}$-measurable. For implementation, $Z_{t_i}^\pi$ will be chosen as a function of $X_{t_i}^\pi$ parameterized by neural networks. We now present a posterior estimation of this simulation error.

\begin{thm}\label{thm:posteriorbdd}
Under Assumptions~\ref{assump:l2regularity} and \ref{assump:posteriorbdd}, we have
$$
\sup_{t \in [0,T]} \hE|Y_t - Y_{\pi\left(t\right)}^\pi|^2+\int_0^T \dbE|Z_t-{Z}_{\pi\left(t\right)}^{\pi}|^2\ud t \leq C\Big[h+\Var\left(Y_0^{\pi}\right)\Big], \nonumber
%\max_{0\leq i\leq n-1} \dbE|Y_{t_i}^\pi-{Y}_{t_i}|^2+\int_0^T \dbE|Z_t-{Z}_{\pi(t)}^{\pi}|^2\ud t \leq C[h+\Var(Y_0^{\pi})] 
$$
where $\left\{X_{t_i}^\pi, Y_{t_i}^\pi, Z_{t_i}^\pi\right\}_{i=0}^n$ follows \eqref{eq:BSDEdiscrete}, and $\left(Y_t, Z_t\right)$ satisfies \eqref{def:BSDE}.
\end{thm}

\begin{proof}

By definitions of $Y_{t_i}$ and $Y_{t_i}^\pi$, the Lipschitz property of $f$ and Cauchy-Schwartz inequality, we have
\begin{align}
&\quad \dbE\left|Y_{t_i}-Y_{t_i}^{\pi}\right|^2 \nonumber\\
&\leq 3\Big(\dbE\left|g\left(X_T\right)-g\left(X_T^{\pi}\right)\right|^2+\dbE|\int_{t_i}^T f\left(s,X_s,Y_s,Z_s\right)-f(\pi(s),X^{\pi}_{\pi\left(s\right)},Y^{\pi}_{\bar \pi\left(s\right)},Z^{\pi}_{\pi \left(s\right)})\ud s|^2 \nonumber\\
&\quad +\dbE|\int_{t_i}^T Z_s-Z_{\pi\left(s\right)}^{\pi} \ud W_s|^2\Big)\nonumber\\
&\leq  3\Big(K^2h+T\dbE\int_0^T|f(s,X_s,Y_s,Z_s)-f(\pi(s),X^{\pi}_{\pi(s)},Y^{\pi}_{\bar \pi\left(s\right)},Z^{\pi}_{\pi \left(s\right)})|^2\ud s \nonumber\\
&\quad +\dbE\int_0^T |Z_s-Z_{\pi\left(s\right)}^{\pi}|^2 \ud s\Big)\nonumber\\
&\leq  3\Big(Ch+T\dbE \int_0^T 2K^2h+ K_y|Y_s - Y_{\bar\pi\left(s\right)}^\pi|^2+K_z|Z_s-Z_{\pi\left(s\right)}^{\pi}|^2 \ud s+\dbE\int_0^T |Z_s-Z_{\pi\left(s\right)}^{\pi}|^2  \ud s\Big) \nonumber\\
&\leq  Ch+3\left(1+TK_z\right)\dbE\int_0^T |Z_s-Z_{\pi\left(s\right)}^{\pi}|^2 \ud s + 6TK_y \dbE\int_0^T |Y_s - Y_{\bar \pi\left(s\right)}|^2 + |Y_{\bar \pi\left(s\right)} - Y_{\bar\pi\left(s\right)}^\pi|^2 \ud s. \nonumber
\end{align}
By Proposition \ref{prop:pathregularity}, the first term in the second integral is bounded by $Ch$, which gives 
\begin{equation}
\dbE|Y_{t_i}-Y_{t_i}^{\pi}|^2\leq Ch+3\left(1+TK_z\right)\dbE\int_0^T |Z_s-Z_{\pi\left(s\right)}^{\pi}|^2 \ud s + 6TK_y \dbE\int_0^T \max_{0\leq j\leq n-1}|Y_{t_j} - Y_{t_j}^\pi|^2 \ud s.  \nonumber
\end{equation}
Maximizing over all $i$ yields
\begin{align}
\max_{0\leq i\leq n-1}\dbE\left|Y_{t_i}-Y_{t_i}^{\pi}\right|^2&\leq Ch+3\left(1+TK_z\right)\dbE\int_0^T |Z_s-Z_{\pi\left(s\right)}^{\pi}|^2 \ud s\\ \nonumber
& \quad +6T K_y \int_0^T \max_{0\leq i\leq n-1}|Y_{t_j} - Y_{t_j}^\pi|^2 \ud s.    \nonumber 
\end{align}
%$$
%\max_{0\leq i\leq n-1}\dbE\left|Y_{t_i}-Y_{t_i}^{\pi}\right|^2\leq Ch+3\left(1+TK_z\right)\dbE\int_0^T \left|Z_s-Z_{\pi\left(s\right)}^{\pi}\right|^2 \ud s\\ \nonumber
%+6T K_y \int_0^T \max_{0\leq i\leq n-1}\left|Y_{t_j} - %Y_{t_j}^\pi\right|^2 \ud s.    \nonumber 
%$$
By Assumption \ref{assump:posteriorbdd}, one has $6T^2K_y <1$, which implies
\begin{equation}
\label{y respect to z}
\max_{0\leq i\leq n-1}\dbE\left|Y_{t_i}-Y_{t_i}^{\pi}\right|^2\leq Ch+\frac{3\left(1+TK_z\right)}{1-6T^2K_y}\dbE\int_0^T |Z_s-Z_{\pi\left(s\right)}^{\pi}|^2 \ud s.
\end{equation}
Now it suffices to show 
$$
\dbE \int_0^T|Z_s-Z_{\pi\left(s\right)}^{\pi}|^2 \ud s \leq C\Big[h+\Var(Y_0^\pi)\Big]. \nonumber
$$
Since $Y_0 \in \cF_0$ is deterministic, it is clear that $\Var\left(Y_0^\pi\right) =\Var\left(Y_0^\pi-Y_0\right)$. we deduce
\begin{align}
\Var\left(Y_0^\pi\right) &= \dbE\Bigg\{g\left(X_T^\pi\right)-g\left(X_T\right)+\int_{0}^T f(\pi(s),X^{\pi}_{\pi(s)},Y^{\pi}_{\bar \pi\left(s\right)},Z^{\pi}_{\pi \left(s\right)})-f\left(s,X_s,Y_s,Z_s\right)\ud s \nonumber\\
&\quad+\int_0^T (Z_s-Z^{\pi}_{\pi \left(s\right)})\ud W_s-\dbE\Big[g\left(X_T^\pi\right)-g\left(X_T\right) \nonumber\\
&\quad  +\int_{0}^T f(\pi(s),X^{\pi}_{\pi\left(s\right)},Y^{\pi}_{\bar \pi\left(s\right)},Z^{\pi}_{\pi \left(s\right)})-f\left(s,X_s,Y_s,Z_s\right)\ud s\Big]\Bigg\}^2. \nonumber
\end{align}
Using the inequality $a^2+2ab \leq\left(a+b\right)^2 $, one has
\begin{align}
&\quad\Var(Y_0^\pi)\nonumber\\
&\geq \dbE\int_0^T|Z_s-Z^{\pi}_{\pi \left(s\right)}|^2\ud s
+2\dbE\Bigg\{\bigg[g\left(X_T^\pi\right)-g\left(X_T\right) +\int_{0}^T f(\pi(s),X^{\pi}_{\pi\left(s\right)},Y^{\pi}_{\bar \pi\left(s\right)},Z^{\pi}_{\pi \left(s\right)})\nonumber\\
&\quad -f(s,X_s,Y_s,Z_s)\ud s-\dbE\Big(g\left(X_T^\pi\right)-g\left(X_T\right)+\int_{0}^T f(\pi(s),X^{\pi}_{\pi\left(s\right)},Y^{\pi}_{\bar \pi\left(s\right)},Z^{\pi}_{\pi \left(s\right)})\nonumber\\
&\quad -f\left(s,X_s,Y_s,Z_s\right)\ud s\Big)\bigg]\int_0^T Z_s-Z^{\pi}_{\pi \left(s\right)}\ud W_s \Bigg\}. \nonumber
\end{align}
For positive arbitrary $\epsilon$, we have $2ab \geq -\epsilon a^2-{1\over \epsilon} b^2$. Together with $\Var\left(X\right)\leq \dbE X^2$, we get:
\begin{align}
 \Var\left(Y_0^\pi\right) &\geq  \left(1-\epsilon\right)\dbE \int_0^T|Z_s-Z^{\pi}_{\pi \left(s\right)}|^2 \ud s-{1\over \epsilon} \dbE\Big[g\left(X_T^\pi\right)-g\left(X_T\right) \nonumber\\
& \quad  +\int_{0}^T f(\pi(s),X^{\pi}_{\pi\left(s\right)},Y^{\pi}_{\bar \pi\left(s\right)},Z^{\pi}_{\pi \left(s\right)})-f\left(s,X_s,Y_s,Z_s\right)\ud s\Big]^2. \nonumber
\end{align}
For positive arbitrary $\delta$, the use of  $\left(a+b\right)^2 \leq \left(1+{1\over\delta}\right)a^2+\left(1+\delta\right)b^2$ leads to
\begin{align}
\Var\left(Y_0^\pi\right)&\geq\left(1-\epsilon\right)\dbE \int_0^T|Z_s-Z^{\pi}_{\pi \left(s\right)}|^2\ud s-{1\over \epsilon}\Bigg\{(1+{1\over \delta}) \dbE\Big[g\left(X_T^\pi\right)-g\left(X_T\right)\Big]^2\nonumber\\
& \quad+\left(1+\delta\right)\dbE \Big[\int_0^T f(\pi(s),X^{\pi}_{\pi\left(s\right)},Y^{\pi}_{\bar \pi\left(s\right)},Z^{\pi}_{\pi \left(s\right)})-f(s,X_s,Y_s,Z_s)\ud s\Big]^2\Bigg\} \nonumber\\
&\geq \left(1-\epsilon\right)\dbE \int_0^T|Z_s-Z^{\pi}_{\pi \left(s\right)}|^2\ud s -{1\over \epsilon}\Bigg[(1+{1\over \delta})K^2h \nonumber\\
&\quad+\left(1+\delta\right)T\dbE \int_0^T |f(\pi(s),X^{\pi}_{\pi\left(s\right)},Y^{\pi}_{\bar \pi\left(s\right)},Z^{\pi}_{\pi \left(s\right)})-f\left(s,X_s,Y_s,Z_s\right)|^2\ud s \Bigg].\nonumber
\end{align}
The Lipschitz condition of $f$ leads to
\begin{align}
&\quad \Var\left(Y_0^\pi\right)\nonumber\\
&\geq  \left(1-\epsilon\right)\dbE \int_0^T|Z_s-Z^{\pi}_{\pi \left(s\right)}|^2\ud s-{1\over \epsilon}\Bigg[(1+{1\over \delta})Ch \nonumber\\
& \quad 
+\left(1+\delta\right)T\dbE\int_0^T K^2h+K^2h+K_y|Y_s-Y_{\bar \pi\left(s\right)}^\pi|^2+K_z|Z_s-Z^{\pi}_{\pi \left(s\right)}|^2\ud s \Bigg] \nonumber\\
&\geq \bigg[1-\epsilon-{\left(1+\delta\right)TK_z\over\epsilon}\bigg]\dbE \int_0^T|Z_s-Z^{\pi}_{\pi \left(s\right)}|^2\ud s - Ch\nonumber\\
& \quad -{2\left(1+\delta\right)\over \epsilon}T^2K_y\max_{0\leq i\leq n-1}\dbE\left|Y_{t_i}-Y_{t_i}^{\pi}\right|^2. \nonumber
%- {1\over \epsilon}[(1+{1\over \delta})C+(1+\delta)T^2k(C+h)]h\\
\end{align}
Using \eqref{y respect to z} brings
$$
    \Var\left(Y_0^\pi\right)
\geq \left[1-\epsilon-{\left(1+\delta\right)TK_z\over\epsilon}-{2\left(1+\delta\right)\over \epsilon} T^2 K_y\frac{3\left(1+TK_z\right)}{1-6T^2K_y}\right]\dbE \int_0^T|Z_s-Z^{\pi}_{\pi \left(s\right)}|^2\ud s - Ch. \nonumber
$$
%by using the following inequality:
%\beaa
%&&,\qq  \forall\epsilon>0\\
%&&a^2+2ab \leq(a+b)^2 \leq (1+{1\over\delta})a^2+(1+\delta)b^2,\qq \forall\delta>0\\
%&&\Var(X)\leq \dbE X^2
%\eeaa
%[{1\over \epsilon}(1+{1\over \delta})C +(1+\delta)Tk(C+h)+{1+\delta\over \epsilon}Ck_y]h
Choose $\epsilon=\sqrt{\left(1+\delta\right)TK_z+2\left(1+\delta\right)T^2 K_y\frac{3\left(1+TK_z\right)}{1-6T^2K_y}}$, since $K_z< {1\over 4T}$ and $K_y$ is sufficiently small, we may choose proper $\delta$ sufficiently small such that
$$
1-2\sqrt{\left(1+\delta\right)TK_z+2\left(1+\delta\right)T^2 K_y\frac{3\left(1+TK_z\right)}{1-6T^2K_y}}>0. \nonumber
$$
% Since $k_y$ is sufficiently small, we can still guarantee 
%\beaa
%&&1-2\sqrt{(1+\delta)Tk_z}-{1+\delta\over \epsilon}Ck_y\\
%&=&1-2\sqrt{(1+\delta)Tk_z}-{1+\delta\over \sqrt{(1+\delta)Tk_z}}Ck_y>0.
%\eeaa
Therefore,
$$
    \max_{0\leq i\leq n-1} \dbE|{Y}_{t_i}- Y_{t_i}^\pi|^2+\int_0^T \dbE|Z_t-{Z}_{\pi(t)}^{\pi}|^2\ud t \leq C\Big[h+\Var\left(Y_0^{\pi}\right)\Big]. \nonumber
$$
Apply once again the path regularity result \eqref{eq:pathregularity}, we obtain the desired result. 
\end{proof}

\begin{rem}
 {The non-adaptness of $Y_{t_i}^{\pi}$ (cf. \eqref{eq:BSDEdiscrete}) makes some estimates in \cite{zhang2004numerical} no longer valid. For example, in the proof of convergence of the numerical solution of BSDE in \cite[Theorem~5.3]{zhang2004numerical}, a crucial step is the following equality:
 $$
 \dbE \Big[(Y_{t_{i-1}}-Y_{t_{i-1}}^{\pi})\int_{t_{i-1}}^{t_i} Z_s-Z_{\pi\left(s\right)}^{\pi} dW_s \Big]=0,
 $$
 which clearly requires the adaptivity of $Y_{t_i}^{\pi}$ to the filtration. }

\end{rem}

The next theorem provides an upper bound for $\Var\left[Y_0^\pi\right]$, which can be sufficiently small subject to universal approximation.
\begin{thm}
\label{easy case for upper bound}
Under Assumptions~\ref{assump:l2regularity}, \ref{assump:FKformula} and  \ref{assump:minimizeVar}{, by} the Nonlinear Feymann-Kac Formula, we have $Z_t=\nabla_x u(t,X_t)\sigma(t,X_t)$. Assume further that the function $\nabla_x u\left(t,x\right) \sigma(t,x)$ is Lipchitz with respect to $x$, with a Lipchitz constant $L$. Then for sufficiently small $h$, we have
$$
\Var\left(Y_0^{\pi}\right)\leq C\Big[h+\sum_{0\leq i\leq n-1}\dbE\left|f_i(X_{t_i}^{\pi})-Z_{t_i}^{\pi}\right|^2h\Big], \nonumber
$$
where $f_i\left(x\right)=\nabla_{x}u(t_i,x) \sigma(t_i,x)$.
\end{thm}

\begin{proof}
Since $Y_0 \in \cF_0$ is deterministic, it is clear that $\Var\left(Y_0^\pi\right) =\Var\left(Y_0^\pi-Y_0\right)$. We have
\begin{align}
\Var\left(Y_0^\pi\right)=& \dbE\Bigg\{g\left(X_T^\pi\right)-g\left(X_T\right)+\int_{0}^T f(\pi(s),X^{\pi}_{\pi\left(s\right)},Y^{\pi}_{\bar \pi\left(s\right)},Z^{\pi}_{\pi \left(s\right)})-f\left(s,X_s,Y_s,Z_s\right)\ud s \nonumber\\
&+\int_0^T (Z_s-Z^{\pi}_{\pi \left(s\right)})\ud W_s -\dbE\Big[g(X_T^\pi)-g(X_T) \nonumber\\
&+\int_{0}^T f(\pi(s),X^{\pi}_{\pi(s)},Y^{\pi}_{\bar \pi\left(s\right)},Z^{\pi}_{\pi \left(s\right)})-f\left(s,X_s,Y_s,Z_s\right)\ud s\Big]\Bigg\}^2. \nonumber
\end{align}
Use the inequality $\left(a+b\right)^2 \leq 2\left(a^2+b^2\right)$ produces
\begin{align}
\Var\left(Y_0^\pi\right)&\leq  2\Bigg[\dbE\int_0^T |Z_s-Z^{\pi}_{\pi \left(s\right)}|^2 \ud s+\Var\bigg(g\left(X_T^\pi\right)-g\left(X_T\right)\nonumber\\
&\quad+\int_{0}^T f(\pi(s),X^{\pi}_{\pi(s)},Y^{\pi}_{\bar \pi\left(s\right)},Z^{\pi}_{\pi \left(s\right)})-f\left(s,X_s,Y_s,Z_s\right)\ud s \bigg) \Bigg].\nonumber
\end{align}
The inequality $\Var\left(X\right)\leq \dbE X^2$ then implies:
\begin{align}
\Var\left(Y_0^\pi\right) \leq& 2\Bigg\{\dbE\int_0^T |Z_s-Z^{\pi}_{\pi\left(s\right)}|^2\ud s+\dbE\bigg[g\left(X_T^\pi\right)-g\left(X_T\right)\nonumber\\
&+\int_{0}^T f(\pi\left(s\right),X^{\pi}_{\pi\left(s\right)},Y^{\pi}_{\bar \pi(s)},Z^{\pi}_{\pi \left(s\right)})-f(s,X_s,Y_s,Z_s)\ud s\bigg]^2\Bigg\}\nonumber\\
\leq& 2\Bigg\{\dbE\int_0^T |Z_s-Z^{\pi}_{\pi \left(s\right)}|^2\ud s+2\Bigg(K^2h\nonumber\\
&+\dbE\Big[\int_{0}^T f(\pi(s),X^{\pi}_{\pi\left(s\right)},Y^{\pi}_{\bar \pi\left(s\right)},Z^{\pi}_{\pi \left(s\right)})-f\left(s,X_s,Y_s,Z_s\right)\ud s\Big]^2\bigg)\Bigg\}.\nonumber
\end{align}
The Lipschitz condition of $f$ and Proposition \ref{prop:pathregularity} then leads to
\begin{align}
&\Var\left(Y_0^\pi\right) \nonumber\\
\leq&  Ch+C\dbE\int_0^T |Z_s-Z^{\pi}_{\pi \left(s\right)}|^2\ud s+2T\dbE\Big[\int_0^T|f(\pi(s),X^{\pi}_{\pi(s)},Y^{\pi}_{\bar \pi\left(s\right)},Z^{\pi}_{\pi (s)})\nonumber\\
&-f\left(s,X_s,Y_s,Z_s\right)|^2 \ud s\Big]\nonumber\\
\leq & Ch+C\dbE\int_0^T |Z_s-Z^{\pi}_{\pi \left(s\right)}|^2\ud s\nonumber\\
& +2T\dbE\Big[\int_0^T Ch+K_y\max_{0\leq i\leq n-1}|Y_{t_i}-Y_{t_i}^{\pi}|^2+K^2|Z_s-Z^{\pi}_{\pi \left(s\right)}|^2\ud s\Big]\nonumber\\
\leq & Ch+C\dbE\int_0^T |Z_s-Z^{\pi}_{\pi \left(s\right)}|^2\ud s+C\Big[\max_{0\leq i\leq n-1}\dbE\left|Y_{t_i}-Y_{t_i}^{\pi}\right|^2\Big].\nonumber
\end{align}
Since Assumption~\ref{assump:minimizeVar} holds, using the result in \eqref{y respect to z}, we have
$$
\Var\left(Y_0^{\pi}\right)\leq  Ch+C\dbE\int_0^T |Z_s-Z^{\pi}_{\pi \left(s\right)}|^2\ud s. \nonumber
$$
Apply the path regularity for $Z$ in Proposition \ref{prop:pathregularity}, we get
\begin{equation}
\label{almost last step}
\Var\left(Y_0^{\pi}\right)\leq C\Big[h+\dbE\sum_{0\leq i\leq n-1}|Z_{t_i}-Z_{t_i}^{\pi}|^2h\Big].
\end{equation}
Observe that
\begin{align}
\dbE\left|Z_{t_i}-Z_{t_i}^{\pi}\right|^2\leq& 2\dbE\left|f_i(X_{t_i}^{\pi})-Z_{t_i}^{\pi}\right|^2+2\dbE\left|Z_{t_i}-f_i(X_{t_i}^{\pi})\right|^2 \nonumber\\
=& 2\dbE\left|f_i(X_{t_i}^{\pi})-Z_{t_i}^{\pi}\right|^2+2\dbE\left|f_i(X_{t_i})-f_i(X_{t_i}^{\pi})\right|^2\nonumber\\
\leq & 2\dbE\left|f_i(X_{t_i}^{\pi})-Z_{t_i}^{\pi}\right|^2+2L^2\dbE\left|X_{t_i}-X_{t_i}^{\pi}\right|^2\nonumber\\
\leq& Ch+2\dbE\left|f_i(X_{t_i}^{\pi})-Z_{t_i}^{\pi}\right|^2. \nonumber
\end{align}
Plug it into \eqref{almost last step}, the desired result follows. 
\end{proof}
\section{Relations to Derivative Pricing} \label{sec:derivative}
In this section, we briefly explain how the solution of BSDE is related to optimal stopping problems through option pricing. 

Let us start with European style options with maturity time $T$ and payoff function $g(x)$. Consider $X_t$ as the price dynamics of $d_1$ stock under risk-neutral measure (\cite[Section~2.8.1]{zhang2017backward}). Denote by $r$ the risk-free interest rate, and define $Y_t$ as
$$
    Y_t = \hE\Big[e^{-r\left(T-t\right)}g\left(X_T\right) \vert \cF_t\Big], \nonumber
$$
which is interpreted as the fair option price at time $t$. Notice that, under proper conditions on $g$ (for instance, at most linear growth), $e^{-rt}Y_t = \hE[e^{-rT} g\left(X_T\right)\vert \cF_t]$ is a martingale and has the representation
$
    \ud (e^{-rt} Y_t) = \tilde Z_t \ud W_t \nonumber
$
for some square-integrable process $\tilde Z_t$. Then by It\^{o}'s formula, one has
$
    \ud Y_t = rY_t \ud t + e^{rt} \tilde Z_t \ud W_t. \nonumber
$
Redefining $e^{rt}\tilde Z_t$ by $Z_t$, we achieve the BSDE system \eqref{def:BSDE} with $f(t, x, y, z) = -ry$. The European option price is $Y_0$, and we shall compute {it} numerically by the method~\eqref{def:BdeepBSDE-goal}--\eqref{def:BdeepBSDE-Yt}.

As mentioned in Remark~\ref{rem:advantage}, the backward deep BSDE method is capable of solving optimal stopping problems, in particular, pricing Bermudan options. Denote $\cT := \left\{\tau_0
= 0, \tau_1, \tau_2, \ldots, \tau_N = T\right\}$ be a collection of pre-determined time stamps on $\left[0,T\right]$. For Bermudan-type options, the buyer has the right to exercise the option and obtain payoff $g\left(X_\tau\right)$ at $\tau \in \cT$. Consequently, the fair price at time $t$ can be written as
$$
    Y_t = \sup_{\tau \in \cT, \tau \geq t} \hE\Big[e^{-r\left(T-\tau\right)}g\left(X_\tau\right) \vert \cF_t\Big]. \nonumber
$$
Or in terms of BSDE, $Y_t$ solves
\begin{equation}\label{def:BSDE-stop}
\begin{dcases}
X_t=x+\int_0^t b(s,X_s)\ud s+\int_0^t \sigma(s,X_s)\ud W_s, \nonumber\\
Y_T = g(X_T), \\
Y_t=Y_{\tau_{j+1}}+\int_t^{\tau_{j+1}} f\left(s,X_s,Y_s,Z_s\right)\ud s-\int_t^{\tau_{j+1}} Z_s\ud W_s, \quad t \in \left(\tau_{j}, \tau_{j+1}\right), \; j = N-1, \ldots, 0, \\
Y_{\tau_{j}} = \max\bigg(g(X_{\tau_{j}}), \; Y_{\tau_{j+1}}+\int_{\tau_{j}}^{\tau_{j+1}} f(s,X_s,Y_s,Z_s)\ud s-\int_{\tau_{j}}^{\tau_{j+1}} Z_s\ud W_s\bigg).
\end{dcases}
\end{equation}
Numerically, to adapt the backward deep BSDE scheme \eqref{def:BdeepBSDE-goal}--\eqref{def:BdeepBSDE-Yt} to optimal stopping problems, we need to replace \eqref{def:BdeepBSDE-Yt} by
\begin{align}
&Y_T^{\pi}=g(X_T^{\pi}),\nonumber\\
& Y_{t_i}^{\pi} = Y^{\pi}_{t_{i+1}} + f(t_i,X_{t_i}^{\pi},Y_{t_{i+1}}^{\pi},Z_{t_i}^{\pi})h - Z_{t_i}^{\pi} \Delta W_i, \quad t_i \in \left(\tau_{j}, \tau_{j+1}\right), \nonumber \\
& Y_{t_i}^{\pi} = \max\bigg(g(X_{t_i}^\pi), \; Y^{\pi}_{t_{i+1}} +
 f(t_i,X_{t_i}^{\pi},Y_{t_{i+1}}^{\pi},Z_{t_i}^{\pi})h - Z_{t_i}^{\pi} \Delta W_i\bigg), \quad t_i = \tau_j, \;j = N-1, \ldots, 0, \nonumber
\end{align}
where for convenience, we've assumed $\tau_j = t_i$ for some $i$, and for all $j$.
\section{Numerical Examples}\label{sec:numerics}
%1. The scheme does not converge for large $k_y$. 

%2. Train loss does not decrease to 0, though $Y_0$ seems accurate (may be due not non fine tuning hyper-parameters).
For numerical illustration, we consider option pricing problems under the Black-Scholes model, under which European options have explicit formulas, and Bermudan options {have} been extensively benchmarked in the literature. Let $X_t$ be the price dynamics of $d_1$ stocks following:
\begin{equation}\label{def:Xt-BS}
    \ud X_t^i/X_t^i = \left(r - \delta_i\right) \ud t + \sigma_i \ud W_t^i, \quad X_0^i = x_0^i, \quad i = 1, 2, \ldots, d_1,
\end{equation}
where $r$ is the risk-free interest rate, $\delta_i$ is the dividend yield, $\sigma_i$ is the volatility and $W$ is a {$d_1$}-dimensional Brownian motion {($d_1=d$)} with correlation $\ud W_t^i \ud W_t^j = \rho_{ij}$.

{Regarding the hyperparameters in our implementation, we follow the choice in \cite{han2018solving} and use a feedforward neural network with 2 hidden layers, each with $d_1+10$ nodes. The activation function we use is ReLU. For European options, we set the batch size (number of paths) be 256. For Bermudan options, we set the batch size (number of paths) be 4096.

We benchmark the performance using the reference intervals from \cite{herrera2021optimal}. The endpoints of the intervals are determined by the upper bound and lower bound from different existing methods, including least squares Monte Carlo \cite{longstaff2001valuing}, deep optimal stopping \cite{becker2019deep}, neural least square Monte Carlo \cite{lapeyre2021neural}, fitted Q-iteration \cite{li2009learning}, and the two algorithms proposed in \cite{herrera2021optimal}: randomized fitted Q-Iteration, randomized least squares Monte Carlo. In particular, in \cite{herrera2021optimal}, they use randomized neural networks with one hidden layer, each with 20 or $\min(20,d_1)$ (for randomized fitted Q-iteration) nodes, and use leaky ReLU or tanh (for randomized recurrent least squares Monte Carlo) as the activation function. For further details on benchmark values, we refer the reader to \cite[Section~6.1]{herrera2021optimal}.
\subsection{Geometric Put}\label{sec:geoput}
We first consider geometric put options as they can be reduced to one-dimensional problems under the Black-Scholes model \eqref{def:Xt-BS}. In this case, the payoff function $g(\cdot)$ reads
$
    g\left(X_T\right) = \Big(K -(\prod_{i=1}^{d_1} X_T^i)^{1/d_1}\Big)^+, \nonumber
$
where $K$ is the strike price. {Since 
\begin{equation*}
\sigma_{ii}(x)=\sigma_i x_i, \  \sigma_{ij}\left(x\right)=0,  1\leq i,j\leq d_1,  i\neq j,
\end{equation*}
the elliptic condition in Assumption~\ref{assump:l2regularity}: $\sigma\left(x\right)\sigma\left(x\right)\transpose \geq \delta I$ for some positive $\delta$ is not satisfied. As discussed in Section \ref{sec:derivative}, $f(t, x, y, z) = -ry$ in this example, which implies $K_y=r$ in Assumption~\ref{assump:posteriorbdd}. In financial markets, $r$ is usually quite small (below we choose $r=2\%$). However, even with this small $r$, the requirement of $r$ is sufficiently small in Assumption~\ref{assump:posteriorbdd} may not be satisfied. On the other hand, when $d_1\geq 2$ and $K>0$, the Lipchitz condition of $g$ on $x$ in Assumption~\ref{assump:l2regularity} is not satisfied when $x$ is close to $(0,0,\cdots,0)$. Nevertheless, our numerical results below still show the convergence.} 

The problem is equivalent to pricing a 1D put option:
$$\ud \hat X_t/ \hat X_t = \hat \mu \ud t + \hat \sigma \ud B_t,\quad  \hat X_0 = \Big(\prod_{i=1}^{d_1} x_0^i\Big)^{1/d_1}, $$
where $\hat \mu = \frac{1}{d_1}\sum_{i=1}^{d_1} \left(r - \delta_i - \half \sigma_i^2\right) + \half \hat \sigma^2$, $\hat \sigma^2 = \frac{1}{d_1^2}\big(\sum_{i=1}^{d_1} \sigma_i^2 + \sum_{i,j = 1}^{d_1} \sigma_i\sigma_j \rho_{ij}\big)$ and $B$ is an one-dimensional Brownian motion.
%$\hat \mu$ and $\hat \sigma$ are given by
%\begin{equation}
%    \hat \mu = \frac{1}{d_1}\sum_{i=1}^{d_1} (r - \delta_i - \half \sigma_i^2) + \half \hat \sigma^2, \quad \hat \sigma^2 = \frac{1}{d_1^2}\left(\sum_{i=1}^{d_1} \sigma_i^2 + \sum_{i,j = 1}^{d_1} \sigma_i\sigma_j \rho_{ij}\right).
%\end{equation}
Therefore, the explicit formula for European geometric put is:
\begin{equation}\label{eq:geometricput}
     P\left(d_1,r, \delta_i, \sigma_i, \rho_{ij}, K, T, x_0\right)=e^{-rT}\Big[K\Phi\left(-d_-\right)-\hat{S}\Phi\left(-d_+\right)\Big],
\end{equation}
where
$$
         \hat{S}=\Big(\prod_{i=1}^{d_1} x_0^i\Big)^{1/d_1}\cdot e^{\hat \mu T},\nonumber
         %\exp\Big((r-{1\over 2d_1}\sum_{i=1}^{d_1} \sigma_i^2+{1\over 2}\hat{\sigma}^2)T\Big), 
         \quad 
     d_+= {\ln\big({\hat{S}\over K}\big)+{1\over 2}\hat{\sigma}^2 T\over \hat{\sigma}\sqrt{T}},  \quad 
     d_-=d_+-\hat{\sigma}\sqrt{T},
$$
and $\Phi\left(\cdot\right)$ denotes the cumulative distribution function of a standard normal distribution. For Bermudan geometric put, we take reference values from \cite[Table~4]{herrera2021optimal}. 

Experiment results are summarized in Table~\ref{tab:geometricput}, where parameters are chosen as:
$$
    r = 2\%, \; \delta_i \equiv 0\%, \; \sigma_i \equiv 20\%, \;  \rho_{ij} \equiv 0, \; x_0^i \equiv 100, \; K = 100, \; T = 1, \; N = 10.
$$
\begin{table}[h]
\caption{European and Bermudan geometric put options for a different number of stocks $d_1$. The analytic value for the European option is calculated using formula \eqref{eq:geometricput} and the range of benchmark value for the Bermudan option is taken from \cite[Table~4]{herrera2021optimal}.}\label{tab:geometricput}
\centering
\begin{tabular}{c c c c| c c}
\toprule
$d_1$ &\multicolumn{3}{c}{European}& \multicolumn{2}{c}{Bermudan} \\ \midrule
 & Price $Y_0$ & Analytic & Relative Err &  Price $Y_0$ & Benchmark \\ %\cmidrule[lr]{2-3}\cmidrule[lr]{4-5}
1 & 6.9317 &6.9359 & -0.06\% & 7.12 & [6.92, 7.08]\\
5 & 3.3058 &3.3105& -0.14\% &3.37 & [3.29, 3.35]\\
10 & 2.3758 & 2.3784& -0.11\% &2.42 & [2.35, 2.40]\\
20 & 1.7033 &1.7009 & 0.14\% &1.73 & [1.59, 1.72]\\
50 & 1.0876 & 1.0867& 0.08\%&1.10&\\
\bottomrule
\end{tabular}
\end{table}

{To see the effect of the choice of the partition size $n$, in Table \ref{tab: convergence for different n}, we present the numerical solution for $d_1=20$, with different choices of $n$.
Recall our partition of size $n$ on $\left[0, T\right]$ is defined by $0 = t_0 < t_1 < \ldots < t_n = T$, $t_i = ih$, $h = T/n$. One can observe that the numerical results get closer to the analytic solution or benchmark interval as $n$ increases (or equivalently $h$ decreases). The convergence of numerical results to the European option price is consistent with our Theorem \ref{thm:posteriorbdd}, and the results are already accurate with relatively small $n$ (or equivalently, large $h$). For Bermudan options, the experiment starts with $n \geq 10$ as we choose $N = 10$ pre-determined possible exercise times. 

%. From Table \ref{tab: convergence for different n}, we also see that the numerical solution for European option converges to the analytical solution even with relatively small $n$ (large $h$). For European option, we set batch size (number of paths) be 256. For Bermudan option, we set batch size (number of paths) be 4096. Since $N=10$, for Bermudan option, $n\geq 10$.

\begin{table}[h]
\caption{European and Bermudan geometric put options for $d_1=20$ stocks and different $n$. The analytic value for the European option is calculated using formula \eqref{eq:geometricput} and the range of benchmark value for the Bermudan option is taken from \cite[Table~4]{herrera2021optimal}.}\label{tab: convergence for different n}
\centering
\begin{tabular}{c c c c| c c}
\toprule
$n$ &\multicolumn{3}{c}{European}& \multicolumn{2}{c}{Bermudan} \\ \midrule
 & Price $Y_0$ & Analytic & Relative Err &  Price $Y_0$ & Benchmark \\ %\cmidrule[lr]{2-3}\cmidrule[lr]{4-5}
 2 & 1.7406 & 1.7009 & 2.33\% &  & \\
 5 &1.7162 & 1.7009 & 0.9\% &  & \\
 10 & 1.7085 & 1.7009 & 0.45\% & 1.82 &[1.59,1.72] \\
20 & 1.7043 & 1.7009 & 0.2\% & 1.79 & [1.59,1.72]\\
50 & 1.7018 &1.7009& 0.05\% & 1.77 &  [1.59,1.72]\\
100& 1.7014 &1.7009& 0.03\% & 1.76 & [1.59,1.72]\\
150 & 1.7002 &1.7009 & -0.04\% & 1.75 & [1.59,1.72]\\
200 &  1.7010 & 1.7009 & 0.01\%&  1.73 & [1.59,1.72]\\
\bottomrule
\end{tabular}
\end{table}}

\subsection{Basket Call}

Next, we test on a basket call option with payoff function 
$
    g(X_T) = \Big(\frac{1}{d_1}\sum_{i=1}^{d_1} X_T^i - K\Big)^+, \nonumber
$
where $X_t$ follows \eqref{def:Xt-BS}. {Similarly as in Section~\ref{sec:geoput}, the elliptic condition in Assumption~\ref{assump:l2regularity} and the smallness requirement of $K_y$ in Assumption~\ref{assump:posteriorbdd} are not satisfied. Nevertheless, we can still obtain convergence results as below. }

%Still, below we see convergence results.}

%\revise{Since 
%\begin{equation*}
%\sigma_{ii}\left(x\right)=\sigma_i x_i, \  \sigma_{ij}\left(x\right)=0,  1\leq i,j\leq d_1,  i\neq j,
%\end{equation*}
%the elliptic condition in Assumption~\ref{assump:l2regularity}: $\sigma\left(x\right)\sigma\left(x\right)\transpose \geq \delta I$ for some positive $\delta$ is not satisfied. As discussed in Section \ref{sec:derivative}, $f(t, x, y, z) = -ry$ in this example, which implies $K_y=r$ in Assumption~\ref{assump:posteriorbdd}. In financial markets, $r$ is usually quite small (below we choose $r=2\%$). Nevertheless, even with this small $r$, the requirement of $r$ is sufficiently small in Assumption~\ref{assump:posteriorbdd} is not satisfied. On the other hand, our numerical results still show the convergence.} 

Here, one can no longer reduce it to a 1D problem. For the European style options, we use the forward deep BSDE method \cite{han2017deep,han2018solving} as a benchmark, and for the Bermudan style options, we use reference values from \cite{herrera2021optimal}. %\revise{The endpoints of the intervals are determined by the upper bound and lower bound from different methods (least squares Monte Carlo, deep optimal stopping, deep neural network, Randomized least squares Monte Carlo, fitted Q-iteration, randomized fitted Q-Iteration) in \cite[Table~4]{herrera2021optimal}.} 
The numerics are summarized in Table~\ref{tab:basketcall}, where parameters are chosen as the same in Table \ref{tab:geometricput}. In both experiments, we are able to see that the performance of the backward deep BSDE algorithm is consistent with the state-of-the-art methods. 
\begin{table}[h]
\caption{European and Bermudan basket call options for a different number of stocks $d_1$. The value $Y_0$ (Forward Scheme) is produced by the forward deep BSDE scheme \cite{han2017deep,han2018solving} and the range of benchmark values for Bermudan option are taken from \cite[Table~3]{herrera2021optimal}}\label{tab:basketcall}
\centering
\begin{tabular}{c c c c| c c}
\toprule
$d_1$ &\multicolumn{3}{c}{European}& \multicolumn{2}{c}{Bermudan} \\ \midrule
 & Price $Y_0$ & $Y_0$ (Forward Scheme) & Relative Err &  Price $Y_0$ & Benchmark \\ %\cmidrule[lr]{2-3}\cmidrule[lr]{4-5}
5 & 4.6343 &4.64& -0.12\% & 4.66 & [4.44, 4.65]\\
10 & 3.6305 & 3.6312& -0.02\% &3.65  & [3.38, 3.63]\\
20 &2.9467 & 2.9477& -0.03\%& 2.96& \\
50 & 2.3840 &2.3859&-0.08\% & 2.40 & [1.66, 2.37]\\
\bottomrule
\end{tabular}
\end{table}
\section{Conclusion} \label{sec:conclusing}

This paper studies the convergence of the deep BSDE method algorithm proposed by Wang \emph{et al.} \cite{wang2018deep}. We show that under appropriate conditions, the posterior error of the numerical solution can be sufficiently small, subject to universal approximations. For numerical experiments, as the back deep BSDE is capable of computing the optimal stopping problem, we present Bermudan option pricing as an example, which is quite common in financial markets. The performance is pretty well and consistent with the state-of-the-art methods. We remark that, although the numerical results work well, the theoretical convergence of the algorithm for optimal stopping problems remains open, and we leave it as future work.

\section*{Acknowledgement}
{The authors would like to thank the valuable comments and suggestions from the anonymous referee.} The authors would like to thank Professor Jianfeng Zhang for fruitful discussions. 
Part of this work was done when C.G. and S.G. were visiting students at the University of California, Santa Barbara, whose hospitality is greatly appreciated.
R.H. was partially supported by the NSF grant DMS-1953035, the Faculty Career Development Award, the Regents' Junior Faculty Fellowship, the Research Assistance Program Award, and the Early Career Faculty Acceleration funding at the University of California, Santa Barbara.

\bibliographystyle{plain}
\bibliography{bibliography.bib}

\end{document}